\UseRawInputEncoding
\documentclass[12pt,reqno]{amsart}

\usepackage[english]{babel} % English language/hyphenation
\usepackage{amsmath,amsfonts,amsthm,bm} % Math packages
\usepackage{verbatim}

\newtheorem{thm}{Theorem}
\newtheorem{prop}{Proposition}
\newtheorem{cor}{Corollary}
\newtheorem{lem}{Lemma}

\theoremstyle{definition}
\newtheorem{rem}{Remark}

\providecommand{\NN}{\mathbb{N}}

\DeclareMathOperator{\spa}{span}

\def\u{\mathbf{u}}
\def\vv{\mathbf{v}}
\def\w{\mathbf{w}}
\def\a{a}
\def\m{\mathbf{m}}
\def\G{\mathbf{G}}
\def\F{\mathbf{F}}
\def\X{\mathbf{X}}
\def\Y{\mathbf{Y}}

\makeatletter
\@namedef{subjclassname@2020}{%
  \textup{2020} Mathematics Subject Classification}
\makeatother

\begin{document}

\title[Prime and semiprime submodules of $R^n$]{Prime and semiprime submodules of $R^n$ and a related Nullstellensatz for $M_n(R)$}
%{A Nullstellensatz for free modules}

\author{J. Cimpri\v c}

\date{February 2nd 2021}

\thanks{This work was supported by grant P1-0222 from Slovenian Research Agency}


\address{University of Ljubljana, Faculty of Mathematics and Physics, Department of Mathematics, Jadranska 21, 1000 Ljubljana, Slovenia}

\email{jaka.cimpric@fmf.uni-lj.si}

\begin{abstract}
Let $R$ be a commutative ring with $1$ and $n$ a natural number.
We say that a submodule $N$ of $R^n$ is \textit{semiprime} if for every
$f=(f_1,\ldots,f_n) \in R^n$ such that $f_i f \in N$ for $i=1,\ldots,n$ we have $f \in N$. 
Our main result is that every semiprime submodule of $R^n$ is equal to the intersection
of all prime submodules containing it. It follows that every semiprime left ideal
of $M_n(R)$ is equal to the intersection of all prime left ideals that contain it.
For $R=k[x_1,\ldots,x_d]$ where $k$ is an algebraically closed field we can rephrase 
this result as a Nullstellensatz for $M_n(R)$: For every $G_1,\ldots,G_m,F \in M_n(R)$,
$F$ belongs to the smallest semiprime
left ideal of $M_n(R)$ that contains $G_1,\ldots,G_m$ iff for every $a \in k^d$ and
$v \in k^n$ such that $G_1(a)v=\ldots=G_m(a)v=0$ we have $F(a)v=0$.

\end{abstract}

\maketitle

\thispagestyle{empty}

\section{Introduction}

The notions of prime and semiprime submodule of an $R$-module
%(of a right module over an associative ring) 
were introduced in \cite{dauns}.
While the definition of a prime submodule has stood the test of time,
it is fair to say that the definition of a semiprime submodule has
been far less successful. The problem is that a semiprime submodule
need not be equal to the intersection of all prime submodules that contain it (see below).
This has been partially circumvented by the so called radical formulas, 
see e.g.  \cite{rad3}, \cite{rad1}, \cite{rad5}, \cite{rad6}, \cite{rad4}, \cite{rad2},
which require strong assumptions on the ring $R$.
%but a major disadvantage of this approach is that it works only for some classes of rings.

The aim of this paper is to propose a new definition of a semiprime submodule and demonstrate its advantages
in ring theory and algebraic geometry. 
Let $R$ be an commutative ring with $1$ and $n \in \NN$. We say that a submodule $N$ of $R^n$ 
is \textit{semiprime} if for every $\mathbf{f}=(f_1,\ldots,f_n) \in R^n$ such that $f_i \mathbf{f} \in N$ for $i=1,\ldots,n$ we have $\mathbf{f} \in N$.
Although we can  extend this definition in several ways, it is not clear how to extend it to submodules of general $R$-modules.

Submodules that are semiprime in the sense of \cite{dauns} will be called weakly semiprime here.
A submodule $N$ of an $R$-module $M$ is \textit{weakly semiprime} if for every $r \in R$
and every $\mathbf{m} \in M$ such that $r^2 \mathbf{m} \in N$ we have $r\mathbf{m} \in N$.
 We will show that every semiprime submodule of $R^n$ is weakly semiprime
but the converse is false (see Lemma \ref{semiweak}). 

The first advantage of the new definition is that every semiprime submodule 
%of every finitely generated free module over a commutative ring with $1$ 
of $R^n$ is equal to the intersection of all prime submodules that contain it (see Theorem \ref{thm2}).
We consider this our main result.

The second advantage of the new definition is that the natural one-to-one correspondence
between submodules of $R^n$ and left ideals of $M_n(R)$ sends semiprime submodules of  $R^n$
into semiprime left ideals of $M_n(R)$ (see Proposition \ref{thm4}). 
A left ideal of $S=M_n(R)$ is \textit{semiprime} iff $xSx \subseteq I$ implies $x \in I$.
%i.e. if $I$ is a semiprime submodule of $S$.)
%A left ideal $I$ of $S=M_n(R)$ is \textit{prime} if $xSy \subseteq I$ implies that $xS \subseteq I$ or $y \in I$.
As an application of our main result we show that 
%for a commutative $R$ with $1$ 
every semiprime left ideal of $M_n(R)$
is equal to the intersection of all prime left ideals that contain it (see Theorem \ref{thm5}).

Our motivation for studying this topic comes from algebraic geometry.
As a special case of our results we obtain a new Nullstellensatz for matrix polynomials (see Theorem \ref{thm6}).
It says that for every matrix polynomials $\G_1,\ldots,\G_m,\F \in M_n(k[x_1,\ldots,x_d])$ 
where $k$ is an algebraically closed field, $\F$ belongs to the smallest semiprime
left ideal of $M_n(k[x_1,\ldots,x_d])$ that contains $\G_1,\ldots,\G_m$ iff for all $\a \in k^d$ and
$\vv \in k^n$ such that $\G_1(\a)\vv=\ldots=\G_m(\a)\vv=0$ we have $\F(\a)\vv=0$.

The main part of the paper is divided into three sections.
In Section \ref{sec2} we give explicit descriptions of prime and maximal submodules of $R^n$.
Section \ref{sec3} contains our main result and Hilbert's Nullstellensatz for $k[x_1,\ldots,x_d]^n$.
Section \ref{sec4} contains the variant of our main result for $M_n(R)$
and one-sided Hilbert's Nullstellensatz for $M_n(k[x_1,\ldots,x_d])$.

\section{Prime submodules of $R^n$}
\label{sec2}

Let $R$ be a commutative ring with $1$ and let $M$ be an $R$ module.
A submodule $N$ of $M$ is \textit{prime} if for every $r \in R$
and every $\m \in M$ such that $r \m \in N$ we 
have that $\m \in N$ or $rM \subseteq N$.
If $N$ is a prime submodule of $M$ then the set
$(N:M):=\{ r \in R \mid r M \subseteq N\}$ is clearly a prime ideal of $R$.
Let $\mathfrak{p}$ be a prime ideal of $R$. We say that a prime submodule $N$ of $M$
is $\mathfrak{p}$-prime if $(N:M)=\mathfrak{p}$. 
If $\mathfrak{p}$ is proper (i.e. $\mathfrak{p} \ne R$)
%and $N$ is a $\mathfrak{p}$-prime submodule of $M$ then $N$ is a proper submodule of $M$ (i.e. $N \ne M$).
then every $\mathfrak{p}$-prime submodule of $M$ is also proper.

For $M=R^n$ there are several explicit descriptions of its prime submodules,
see e.g. \cite{class3},  \cite{class2}, \cite{class1}.
Our take on this is Proposition \ref{thm1}.

\begin{prop}
\label{thm1} 
Let $R$ be a commutative ring with $1$, $\mathfrak{p}$ a proper prime ideal of $R$ and $k$ the field of fractions of $R/\mathfrak{p}$.
\begin{enumerate}
\item There is a natural one-to-one correspondence between $\mathfrak{p}$-prime submodules of $R^n$ and proper linear subspaces of $k^n$.
\item For every submodule $N$ of $R^n$, the smallest $\mathfrak{p}$-prime submodule containing $N$ is equal to
$$K(N,\mathfrak{p}):=\{\mathbf{r} \in R^n \mid \exists c \in R \setminus \mathfrak{p} \colon c \mathbf{r} \in N+\mathfrak{p}^n\}.$$
\end{enumerate}
\end{prop}

\begin{proof}
Let $\phi \colon R \to k$ be the natural ring homomorphism and let $\bm{\phi}$ be the ring homomorphism
$\phi \times \ldots \times \phi \colon R^n \to k^n$.
We would like to show that the mappings $W \mapsto \bm{\phi}^{-1}(W)$ and $N \to \spa \bm{\phi}(N)$ give a one-to-one correspondence between 
the proper linear subspaces of $k^n$ and $\mathfrak{p}$-prime modules of $R^n$. 

%For every submodule $N$ of $R^n$, the span of the set $\bm{\phi}(N)$ is clearly a vector subspace of $k^n$.
Let $W$ be a proper linear subspace of $k^n$. Clearly, the set $\bm{\phi}^{-1}(W)$ is a submodule of $R^n$. To show that it is prime,
pick $r \in R$ and $\mathbf{m} \in R^n$
such that $r \mathbf{m} \in \bm{\phi}^{-1}(W)$. It follows that $\phi(r) \bm{\phi}(\mathbf{m})=\bm{\phi}(r \mathbf{m}) \in W$.
Since $W$ is a linear subspace of $k^n$, it follows that either $\phi(r)=0$ or $\bm{\phi}(\mathbf{m}) \in W$.
In the first case, $r R^n \subseteq \bm{\phi}^{-1}(W)$ and in the second case $\mathbf{m} \in \bm{\phi}^{-1}(W)$.
%Therefore $\bm{\phi}^{-1}(W)$ is a prime submodule of $R^n$.
To show that $(\bm{\phi}^{-1}(W) \colon R^n)=\mathfrak{p}$ pick any $a \in R$ and note that $a \in (\bm{\phi}^{-1}(W) \colon R^n)$ iff 
$a R^n \subseteq \bm{\phi}^{-1}(W)$ iff $\phi(a) \mathbf{e}_i \in W$ for every $i=1,\ldots,n$ iff $\phi(a)=0$ or $W=k^n$.
Since $W$ is proper, this is equivalent to $\phi(a)=0$, i.e. $a \in \phi^{-1}(0)=\mathfrak{p}$.

The proof that $\bm{\phi}^{-1}(\spa \bm{\phi}(N))=N$ for every prime submodule $N$ will be split into two parts.
Firstly, we will show that $\bm{\phi}^{-1}(\spa \bm{\phi}(N))=K(N,\mathfrak{p})$ for every submodule $N$ of $R^n$.
For every $\mathbf{r} \in \bm{\phi}^{-1}(\spa \bm{\phi}(N))$ there exist $\alpha_1,\ldots,\alpha_m \in k$ and $\mathbf{n}_1,\ldots,\mathbf{n}_m \in N$
such that $\bm{\phi}(\mathbf{r})=\alpha_1 \bm{\phi}(\mathbf{n}_1)+\ldots+\alpha_m \bm{\phi}(\mathbf{n}_m)$. Pick $c_1,\ldots,c_m \in R$ and $c \in R \setminus \mathfrak{p}$
such that $\alpha_i=\phi(c_i) \phi(c)^{-1}$ for every $i=1,\ldots,m$. Clearly, $\sum_{i=1}^m c_i \mathbf{n}_i \in N$ and $c \, \mathbf{r}-\sum_{i=1}^m c_i \mathbf{n}_i \in \bm{\phi}^{-1}(0)=\mathfrak{p}^n$.
Secondly, we claim that $K(N,\mathfrak{p})=N$ if $N$ is a $\mathfrak{p}$-prime submodule. By the definition of $\mathfrak{p}$, we have that $\mathfrak{p} \mathbf{e}_i \subseteq N$ for every $i=1,\ldots,m$, which implies that $\mathfrak{p}^n \subseteq N$. 
It follows that for every $\mathbf{n} \in K(N,\mathfrak{p})$ there exists $c \in R \setminus \mathfrak{p}$ such that $c \mathbf{n} \in N$. Since $N$ is $\mathfrak{p}$-prime, it follows that $\mathbf{n} \in N$. This proves the claim.

It remains to show that $\spa \bm{\phi}(\bm{\phi}^{-1}(W))=W$ for every linear subspace $W$ of $k^n$. 
Pick any $\w \in W$. There exist $c_1,\ldots,c_n \in R$ and $c \in R \setminus \mathfrak{p}$
such that $w_i=\phi(c_i)\phi(c)^{-1}$ for every $i=1,\ldots,n$. 
Write $\mathbf{u}=(c_1,\ldots,c_n)$ and note that $\mathbf{u} \in \bm{\phi}^{-1}(W)$. 
It follows that $\w=\phi(c)^{-1} \bm{\phi}(\mathbf{u}) \in \spa \bm{\phi}(\bm{\phi}^{-1}(W))$. The opposite inclusion is clear.

Claim (2) follows from claim (1) and $\bm{\phi}^{-1}(\spa \bm{\phi}(N))=K(N,\mathfrak{p})$. 
\end{proof}

For every ideal $\mathfrak{a}$ of $R$ and every element $\mathbf{u}=(u_1,\ldots,u_n)\in R^n$ we define a submodule $C_{\mathfrak{a},\mathbf{u}}$ of $R^n$ by
$$C_{\mathfrak{a},\mathbf{u}} := \{(r_1,\ldots,r_n) \in R^n \mid \sum_{i=1}^n r_i u_i  \in \mathfrak{a}\}.$$

\begin{cor}
\label{cor1}
Ler $R$ be a commutative ring with $1$. Let $\mathfrak{p}$ be a proper prime ideal of $R$.
\begin{enumerate}
\item A subset of $R^n$ is a maximal $\mathfrak{p}$-prime submodule iff it is of the form 
$C_{\mathfrak{p},\mathbf{u}}$ where $\mathbf{u} \in R^n \setminus \mathfrak{p}^n$.
\item Every $\mathfrak{p}$-prime submodule of $R^n$ is equal to the intersection of
all maximal $\mathfrak{p}$-prime submodules that contain it.
\end{enumerate}
\end{cor}

\begin{proof}
To prove claim (1), write $k$ for the field of fractions of $R/\mathfrak{p}$ and $\phi$ for the natural mapping from $R$ to $k$.
By Proposition \ref{thm1}, a subset of $R^n$ is a maximal $\mathfrak{p}$-prime submodule iff it is of the form
$(\phi \times \ldots \times \phi)^{-1}(W)$ for a maximal proper linear subspace (i.e. a hyperplane) $W$ of $k^n$.
Clearly, $W$ is a hyperplane in $k^n$ iff there exist $\alpha_1,\ldots,\alpha_n \in k$ 
such that $W=\{(z_1,\ldots,z_n) \in k^n \mid \sum_{i=1}^n \alpha_i z_i=0\}$ and $\alpha_i\ne 0$ for at least one $i$. 
Pick $u_1,\ldots,u_n \in R$ and $c \in R \setminus \mathfrak{p}$ such that $\alpha_i=\phi(u_i)\phi(c)^{-1}$ 
for every $i$ and so $u_i \not\in \mathfrak{p}$ for at least one $i$.
For every $(r_1,\ldots,r_n) \in R^n$ we have that $(r_1,\ldots,r_n) \in (\phi \times \ldots \times \phi)^{-1}(W)$ 
iff $(\phi(r_1),\ldots,\phi(r_n)) \in W$ iff $\sum_{i=1}^n \phi(u_i) \phi(r_i)=0$ iff 
$\sum_{i=1}^n u_i r_i \in \mathfrak{p}$ iff $(r_1,\ldots,r_n) \in C_{\mathfrak{p},\mathbf{u}}$.

Claim (2) follows from Proposition \ref{thm1} and the observation that every proper linear subspace of $k^n$ is equal 
to the intersection of all hyperplanes that contain it.
\end{proof}

Our next goal is to characterize maximal submodules of $R^n$.
The relations between maximal and prime submodules of general $R$-modules were already discussed in \cite[Propositions 2-3,2-4,2-5]{maximal}. 

\begin{prop}
\label{cor2}
Let $R$ be a commutative ring with $1$.
\begin{enumerate}
\item A submodule of $R^n$ is maximal iff it is of the form $C_{\mathfrak{m},\mathbf{v}}$ where
$\mathfrak{m}$ is a maximal ideal of $R$ and $\mathbf{v} \in R^n \setminus \mathfrak{m}^n$.
\item If $R$ is a Jacobson ring, then every prime submodule of $R^n$ is equal to the intersection of
all maximal submodules of $R^n$ that contain it.
\end{enumerate}
\end{prop}

\begin{proof}
To prove claim (1), pick a maximal submodule $N$ of $M:=R^n$. By \cite[Proposition 2-5]{maximal}
\footnotemark \footnotetext{Here is a summary of the proof.
If $N$ is a maximal submodule of $M$ then $(N:M)$ is a maximal ideal of $R$ (since $M/N$ is isomorphic to $R/(N:M)$ for cyclic $M/N$.)
On the other hand, if $(N:M)$ is a maximal ideal, then $N$ is a prime submodule of $M$ (since $M/N$ is a torsion-free $R/(M:N)$-module.)}, 
$N$ is a prime submodule of $M$ and the set $\mathfrak{m}:=(M:N)$ is a maximal ideal of $R$ 
By Corollary \ref{cor1}, it follows that $N$ is of the form $C_{\mathfrak{m},\u}$ where $\u \in R^n \setminus \mathfrak{m}^n$.

To prove the converse pick a maximal ideal $\mathfrak{m}$ of $R$ and $\u \in R^n \setminus \mathfrak{m}^n$. For every proper submodule $N$
of $R^n$ which contains $C_{\mathfrak{m},\u}$, we have that $R \ne (N:R^n) \supseteq (C_{\mathfrak{m},\u}:R^n)=\mathfrak{m}$. It follows that $(N:R^n)=\mathfrak{m}$.
Since $C_{\mathfrak{m},\u}$ is a maximal $\mathfrak{m}$-prime submodule of $R^n$, it follows that $N=C_{\mathfrak{m},\u}$.

Claim (2) follows from the observation that $\mathfrak{p}=\bigcap_{\mathfrak{p} \subseteq \mathfrak{m}} \mathfrak{m}$
implies that $C_{\mathfrak{p},\u}=\bigcap_{\mathfrak{p} \subseteq \mathfrak{m}} C_{\mathfrak{m},\u}$.
Namely, pick a prime submodule $N$ of $M$ and write $\mathfrak{p}=(N:M)$. By Corollary \ref{cor1}, it follows that 
$$N=\bigcap_{N \subseteq C_{\mathfrak{p},\u}} C_{\mathfrak{p},\u}
= \bigcap_{N \subseteq C_{\mathfrak{p},\u}} \bigcap_{\mathfrak{p} \subseteq \mathfrak{m}} C_{\mathfrak{m},\u}
= \bigcap_{N \subseteq C_{\mathfrak{m},\u}} C_{\mathfrak{m},\u}$$
where $C_{\mathfrak{m},\u}$ are maximal submodules by claim (1).
\end{proof}

\begin{cor}
\label{cor3}
Let $R=k[x_1,\ldots,x_d]$ where $k$ is an algebraically closed field.
\begin{enumerate}
\item For every $\a \in k^d$ and $(v_1,\ldots,v_n)\in k^n \setminus 0^n$ the submodule
$$\{(f_1,\ldots,f_n) \in R^n \mid \sum_{i=1}^n f_i(\a)v_i=0\}$$
is maximal and every maximal submodule of $R^n$ is of this form.
%$\{(f_1,\ldots,f_n) \in R^n \mid \sum_{i=1}^n f_i(\a)v_i=0\}$ where $\a \in k^d$ and $(v_1,\ldots,v_n)\in k^n \setminus 0^n$.
\item Every prime submodule of $R^n$ is equal to the intersection of
all maximal submodules of $R^n$ that contain it.
\end{enumerate}
\end{cor}

\begin{proof}
Let $R$ be as above.
Hilbert's Nullstellensatz implies that $R$ is a Jacobson ring and every maximal ideal of $R$ is of the form 
$\mathfrak{m}_\a=(x_1-a_1,\ldots,x_d-a_d)$ where $\a=(a_1,\ldots,a_d) \in k^d$.
By Proposition \ref{cor2}, every maximal submodule of $R^n$ is of the form $C_{\mathfrak{m}_\a,\u}$
for some $\a \in R^d$ and $\u \in R^d \setminus \mathfrak{m}_\a^d$ and every prime submodule of $R^n$ is equal to the 
intersection of all maximal submodules containing it.

Since $C_{\mathfrak{m}_\a,\u}=C_{\mathfrak{m}_\a,\u(\a)}$, we may restrict ourselves to $\u \in k^n$. %\setminus (0,\ldots,0)$. 
Thus maximal submodules of $R^n$ are of the form $$C_{\mathfrak{m}_\a,\vv}=\{(f_1,\ldots,f_n) \in R^n \mid \sum_{i=1}^n f_i(a)v_i=0\}$$ 
where $\a \in k^d$ and $\vv =(v_!,\ldots,v_n) \in k^n$ is nonzero. %\setminus (0,\ldots,0)$. 
\end{proof}

Corollary \ref{cor3} can also be extended to finitely generated algebras over algebraically closed fields,
i.e, to rings of the form $R=k[x_1,\ldots,x_d]/J$.

\section{Semiprime submodules of $R^n$}
\label{sec3}

Let $R$ be a commutative ring with $1$ and $n \in \NN$.  
We say that a submodule $N$ of $R^n$ is \textit{semiprime} if for every 
$\mathbf{f}=(f_1,\ldots,f_n) \in M$ such that $f_i \mathbf{f} \in N$
for $i=1,\ldots,n$ we have $\mathbf{f} \in N$.
%The smallest semiprime submodule of $M$ which contains a given submodule $N$ of $M$ will be denoted by $\sqrt{N}$.

\begin{lem}
\label{pissemip}
Every prime submodule of $R^n$ is semiprime.
\end{lem}

\begin{proof}
Suppose that $P$ is a prime submodule of $M=R^n$.
Pick an element $\mathbf{f} \in M$ such that $f_i \mathbf{f} \in P$ for all $i$. If $\mathbf{f} \not\in P$,
then by the definition of a prime submodule $f_i M \subseteq P$ for all $i$. It follows that $f_i \mathbf{e}_i \in P$
for all $i$ where $\mathbf{e}_i$ is the $i$-th standard vector in $M$. Therefore $\mathbf{f}=\sum_{i=1}^n f_i \mathbf{e}_i \in P$.
\end{proof}

Let us compare our definition of a semiprime submodule with the usual definition, see e.g. \cite{semiprime}.
A submodule $N$ of an $R$-module $N$ is \textit{weakly semiprime} (i.e. semiprime in the usual sense)
if for every $r \in R$ and every $\m \in M$ such that $r^2 \m \in N$ we have $r \m \in N$.

\begin{lem} 
\label{semiweak}
Every semiprime submodule of $R^n$ is weakly semiprime but the converse is false in general.
\end{lem}

\begin{proof}
Suppose that $N$ is a semiprime submodule of $R^n$. Pick $r \in R$ and $\mathbf{m} \in R^n$ such that $r^2 \mathbf{m} \in N$.
For $\mathbf{n}=r\mathbf{m}$ we have $n_i \mathbf{n}= m_i(r^2 \mathbf{m}) \in N$ for all $i=1,\ldots,n$. 
By the definition of a semiprime submodule, it follows that $\mathbf{n} \in N$. Therefore $N$ is  weakly semiprime.

Let $R=k[x,y]$ where $k$ is a field and let $N$ be the submodule of $R^2$ generated by $(x^2,xy)$ and $(xy,y^2)$.
Clearly, $N$ consists of all elements $t(x,y)$ where $t$ belongs to the ideal $\langle x,y \rangle$ of $R$ generated by $x$ and $y$. 
Since $(x,y) \not\in N$ while $x(x,y) \in N$ and $y(x,y) \in N$, $N$ is not semiprime.
To show that $N$ is weakly semiprime, pick $r \in R$ and $(m_1,m_2) \in R^2$ such that $r^2(m_1,m_2) \in N$.
Pick also $s \in \langle x,y \rangle$ such that $r^2(m_1,m_2)=s(x,y)$. %From $r^2 m_1=sx$ and $r^2 m_2=sy$,
Note that $r^2$ divides both $sx$ and $sy$ which implies that $r^2$ divides $s$, i.e. $s=kr^2$ for some $k \in R$. 
Since $\langle x,y \rangle$ is a maximal ideal and $kr^2 \in \langle x,y \rangle$, it follows that $kr \in \langle x,y \rangle$. 
Therefore $r(m_1,m_2)=kr(x,y)$ belongs to $N$.
\end{proof}

For every submodule $N$ of an $R$-module $M$ we define its \textit{radical} $\sqrt{N}$ as the intersection of all
prime submodules of $M$ that contain $N$. Since every $\mathfrak{p}$-prime submodule that contains $N$ also contains $K(N,\mathfrak{p})$
$$\sqrt{N}=\bigcap_{\mathfrak{p} \in \mathrm{Spec}(R)} K(N,\mathfrak{p})$$
If $\sqrt{N}=N$ we say that $N$ is a \textit{radical submodule}.

\begin{thm}
\label{thm2}
Every semiprime submodule of $R^n$ is a radical submodule. More precisely, for every submodule $N$ of $R^n$,
its radical $\sqrt{N}$ is the smallest semiprime submodule that contains $N$. 
\end{thm}

\begin{proof}
The second part of the theorem follows from the first part and Lemma \ref{pissemip}.
To prove the first part, pick a semiprime submodule $N$ of $R^n$ and 
an element $\mathbf{f}$ of $R^n$ that does not belong to $N$.
We want to construct a prime submodule of $R^n$  that contains $N$ but avoids $\mathbf{f}$.
The proof will be divided into several steps.

\smallskip

\textbf{Step 1:} For every $t \in \NN$ there is $i \in \{1,\ldots,n\}$ such that $f_i^{2^t} \mathbf{f} \not\in N$.

\smallskip

If $f_i^{2^t} \mathbf{f} \in N$ for all $i$ then $(f_i^{2^{t-1}} \mathbf{f})_j f_i^{2^{t-1}} \mathbf{f}=f_j f_i^{2^t}\mathbf{f} \in N$ for all $i,j$.
Since $N$ is semiprime, it follows that $f_i^{2^{t-1}} \mathbf{f} \in N$ for every $i$. By induction, we get $\mathbf{f} \in N$.

\smallskip

If there exists $i \in \{1,\ldots,n\}$ and a prime ideal $\mathfrak{p}$ such that 
$N \subseteq C_{\mathfrak{p},\mathbf{e}_i}$ and $\mathbf{f} \not\in C_{\mathfrak{p},\mathbf{e}_i}$
then the conclusion of the theorem is true. Suppose now that this is false. 
Then every prime ideal of $R$ which contains the $i$-th component $g_i$ of every element $\mathbf{g} \in N$ also contains $f_i$. 
It follows that for every $i$ there exists $t_i\in \NN$ such that $f_i^{t_i} \in \sum_{\mathbf{g} \in N} Rg_i$.

\smallskip

\textbf{Step 2:} There exists $\mathbf{g} \in N$ and $i \in \{1,\ldots,n\}$ such that $g_i \mathbf{f} \not\in N$.

\smallskip

 If $g_i \mathbf{f} \in N$ for every $i$
and every $\mathbf{g} \in N$ then $f_i^{t_i} \mathbf{f} \in \sum_{\mathbf{g} \in N} Rg_i \mathbf{f} \subseteq N$ for every $i$. By Step 1, it follows that $\mathbf{f} \in N$.

\smallskip

Let $\kappa_i \colon R^{n-1} \to R^n$ be the embedding $(x_1,\ldots,x_{i-1},x_i,\ldots,x_{n-1}) \mapsto (x_1,\ldots,x_{i-1},0,x_i,\ldots,x_{n-1})$.
Clearly,  $N' := \kappa_i^{-1}(N)$ is a semiprime submodule of $R^{n-1}$.
Since $g_i \mathbf{f} \not\in N$ and $\mathbf{g} \in N$, it follows that $g_i \mathbf{f}-f_i \mathbf{g} \not\in N$ and so $\mathbf{f}':=\kappa_i^{-1}(g_i \mathbf{f}-f_i \mathbf{g}) \not\in N'$. By the induction 
hypothesis, there exists a prime ideal $\mathfrak{p}$ of $R$ such that $\mathbf{f}'\not\in K(N',\mathfrak{p})$. 

\smallskip

\textbf{Step 3:} $\mathbf{f} \not\in K(N,\mathfrak{p})$.

\smallskip

If $\mathbf{f} \in K(N,\mathfrak{p})$ then $c \,\mathbf{f} =\mathbf{n}+\mathbf{s}$ for some $c \in R \setminus \mathfrak{p}$, $\mathbf{n} \in N$ and $\mathbf{s} \in \mathfrak{p}^n$.
It follows that $c(g_i \mathbf{f}-f_i \mathbf{g})=(g_i \mathbf{n}-n_i \mathbf{g})+(g_i \mathbf{s}-s_i \mathbf{g}) \in N+\mathfrak{p}^n$ where all brackets have
the $i$-th component equal to zero. Applying $\kappa_i^{-1}$ we get $c \mathbf{f}' \in N' + \mathfrak{p}^{n-1}$. Therefore $\mathbf{f}' \in K(N',\mathfrak{p})$, a contradiction.
\end{proof}

For every $\mathbf{g}=(g_1,\ldots,g_n) \in k[x_1,\ldots,x_d]^n$, $\vv=(v_1,\ldots,v_n) \in k^n$ and $a \in k^d$ we write $\mathbf{g}(a)\vv=\sum_{i=1}^n g_i(a)v_i$.

\begin{thm}
\label{thm3}
Let $R=k[x_1,\ldots,x_d]$ where $k$ is an algebraically closed field. 
For every $\mathbf{g}_1,\ldots,\mathbf{g}_m,\mathbf{f} \in R^n$ the following are equivalent.
\begin{enumerate}
\item $\mathbf{f}$ belongs to the smallest semiprime submodule of $R^n$ that contains $\mathbf{g}_1,\ldots,\mathbf{g}_m$.
\item For every $a \in k^d$ and  $\vv \in k^n$ such that $\mathbf{g}_1(a)\vv=\ldots=\mathbf{g}_m(a)\vv=0$ we have $\mathbf{f}(a)\vv=0$.
\end{enumerate}
\end{thm}

\begin{proof}
Consider the following claims.
\begin{enumerate}
\item[(i)] Every semiprime submodule of $R^n$ that contains $\mathbf{g}_1,\ldots,\mathbf{g}_m$, also contains $\mathbf{f}$.
\item[(ii)] Every prime submodule of $R^n$ that contains $\mathbf{g}_1,\ldots,\mathbf{g}_m$, also contains $\mathbf{f}$.
\item[(iii)] Every submodule of the form $C_{\mathfrak{m}_a,v}$ 
that contains $\mathbf{g}_1,\ldots,\mathbf{g}_m$, also contains $\mathbf{f}$.
\end{enumerate}
Claim (i) rephrases claim (1) and claim (iii) rephrases claim (2). 
By Lemma \ref{pissemip}, (i) implies (ii). By Theorem \ref{thm2}, (ii) implies (i). 
Since each $C_{\mathfrak{m}_a,v}$ is prime, (ii) implies (iii).
By Corollary \ref{cor3}, (iii) implies (ii).
\end{proof}

\section{Left prime and semiprime ideals of $M_n(R)$}
\label{sec4}

Let $R$ be a ring with $1$ and $n \in \NN$. We drop the assumption that $R$ is commutative because it does not simplify the proofs in this section.

There exists a one-to-one correspondence between submodules of $R^n$ and left ideals of $M_n(R)$. 
For every left ideal $I$ of $M_n(R)$ and every submodule $N$ of $R^n$ we define the sets
$$\Phi(I)=\{\mathbf{m} \in R^n \mid \left[ \begin{array}{cc} \mathbf{m} \\ \mathbf{0} \\[-1mm] \vdots \\ \mathbf{0} \end{array} \right] \in I\}
\quad \text{ and } \quad
\Psi(N)=\left[ \begin{array}{cc} N  \\ \vdots \\[2mm] N \end{array} \right].$$
Note that $\Phi(I)$ is a submodule of $R^n$ which consists of all rows of all elements of $I$
and that $\Psi(N)$ is a left ideal of $M_n(R)$ which consists of all matrices that have all rows in $N$.
A short computation shows that $\Psi(\Phi(I))=I$ for every  $I$ and $\Phi(\Psi(N))=N$ for every $N$.

\begin{comment}
For every left ideal $I$ of $M_n(R)$ we write $\Phi(I)$ for the set of all rows of all elements of $I$.  Since
$$\left[ \begin{array}{c} \mathbf{x}_1 \\ \mathbf{x}_2 \\[-1mm] \vdots \\ \mathbf{x}_n \end{array} \right]  \in I
\quad \text{ iff } \quad
\left[ \begin{array}{c} \mathbf{x}_1 \\ \mathbf{0} \\[-1mm] \vdots \\ \mathbf{0} \end{array} \right] \in I,
\left[ \begin{array}{c} \mathbf{x}_2 \\ \mathbf{0} \\[-1mm] \vdots \\ \mathbf{0} \end{array} \right] \in I,
\ldots
\left[ \begin{array}{c} \mathbf{x}_n \\ \mathbf{0} \\[-1mm] \vdots \\ \mathbf{0} \end{array} \right] \in I$$
we can simplify the definition of $\Phi(I)$ to
\end{comment}

\begin{rem}
\label{rem3}
For every left ideal $\mathfrak{a}$ of $R$ and every $\mathbf{u}=(u_1,\ldots,u_n) \in R^n$, the submodule
$$C_{\mathfrak{a},\mathbf{u}}=\{\mathbf{r} \in R^n \mid \sum_{i=1}^n r_i u_i \in \mathfrak{a}\}$$
corresponds to the left ideal 
$$D_{\mathfrak{a},\mathbf{u}}=\{\X \in M_n(R) \mid \X \mathbf{u} \in \mathfrak{a}^n\}.$$
\end{rem}

\begin{rem}
\label{rem4}
Since $\Phi$ and $\Psi$ preserve inclusions, maximal submodules of $R^n$ correspond to maximal left ideals of $M_n(R)$.
By \cite[Theorem 1.2]{maximal2}
\footnotemark \footnotetext{Here is a summary of the proof. If $I$ is a maximal left ideal of $M_n(R)$ then, by Morita theory, 
the simple $M_n(R)$-module $M_n(R)/I$ is isomorphic to $E^n$ for some simple $R$-module $E$. Let $\mathfrak{m}$ be a maximal ideal of $R$ such that 
$E \cong R/\mathfrak{m}$ and let $u+\mathfrak{m}^n$ be the image of $1+I$ under the map $M_n(R)/I \cong (R/\mathfrak{m})^n \cong R^n/\mathfrak{m}^n$.},
every maximal left ideal of $M_n(R)$ is of the form
$D_{\mathfrak{m},\mathbf{u}}$ where $\mathfrak{m}$ is a maximal left ideal of $R$ and $\mathbf{u} \in R^n \setminus \mathfrak{m}^n$.
Therefore, every maximal submodule of $R^n$ is of the form $C_{\mathfrak{m},\mathbf{u}}$ for maximal $\mathfrak{m}$ and
$\mathbf{u} \in R^n \setminus \mathfrak{m}^n$. This is a noncommutative version of Proposition \ref{cor2}.
\end{rem}

A submodule $N$ of a left $R$-module $M$ is \textit{prime} if for every $r \in R$ and $\mathbf{m} \in M$ such that
$r R \mathbf{m} \subseteq N$ we have $\mathbf{m} \in N$ or $rM \subseteq N$.
A submodule $N$ of a left $R$-module $R^n$ is \textit{semiprime} if for every $\mathbf{m}=(m_1,\ldots,m_n) \in R^n$ 
such that $m_i R \mathbf{m} \subseteq N$ for $i=1,\ldots,n$ we have $\mathbf{m} \in N$.
A left ideal $I$ of $R$ is \textit{prime} if it is a prime submodule of $R$, i.e. if $xRy \subseteq I$ implies that $xR \subseteq I$ or $y \in I$.
Similarly, $I$ is \textit{semiprime} if it is a semiprime submodule of $R$, i.e. if $xRx \subseteq I$ implies  $x \in I$.
%A left ideal $I$ of $R$ is \texit{prime} (\textit{semiprime}) if it is prime (semiprime) as a submodule $R$, 
%i.e. if $xRy \subseteq I$ implies that $xR \subseteq I$ or $y \in I$ (if $xRx \subseteq I$ implies that $x \in I$).
%Let $S$ be an associative but not necessarily commutative ring.
%A left ideal $I$ of  $S$ is \textit{prime} if $xSy \subseteq I$ implies that either $y \in I$ or $x S \subseteq I$.
%(In this case the set $(I:S)$ is a two-sided prime ideal of $S$.)
%A left ideal $I$ of  $S$ is \textit{semiprime} if $xSx \subseteq I$ implies that $x \subseteq I$.

\begin{prop}
\label{thm4}
In the bijective correspondence from above prime submodules of $R^n$ correspond to prime left ideals of $M_n(R)$
and semiprime submodules of $R^n$ correspond to semiprime left ideals of $M_n(R)$.
\end{prop}

\begin{proof} 
Suppose that $I$ is a prime left ideal of $M_n(R)$. To show that $\Phi(I)$ is a prime submodule 
pick $x \in R$ and $\textbf{m} \in R^n$ such that $x R \mathbf{m} \subseteq \Phi(I)$.
%By definition, every element from $x R \mathbf{m}$ is the first row of some matrix $\Y \in I$. 
It follows that
$$x M_n(R) \left[ \begin{array}{cc} \mathbf{m} \\ \mathbf{0} \\ \vdots \\ \mathbf{0} \end{array} \right]
\subseteq  \left[ \begin{array}{cc} x R \mathbf{m} \\ x R \mathbf{m} \\ \vdots \\ x R \mathbf{m} \end{array} \right]
\subseteq \left[ \begin{array}{cc} \Phi(I) \\ \Phi(I) \\ \vdots \\ \Phi(I)  \end{array} \right]
= \Psi (\Phi(I))=I.
$$
It follows that either $\left[ \begin{array}{cc} \mathbf{m} \\ \mathbf{0} \\ \vdots \\ \mathbf{0} \end{array} \right] \in I$ or $x M_n(R) \subseteq I$.
Therefore, either $\textbf{m} \in \Phi(I)$ or $x R^n \subseteq \Phi(I)$.

Conversely, let $N$ be a prime submodule of $R^n$. Pick $\X,\Y \in M_n(R)$ such that $\X M_n(R) \Y \subseteq \Psi(N)$.
For every $i,j=1,\ldots,n$ and $r \in R$,
$$\left[ \begin{array}{c} x_{1,i}  r \mathbf{y}_j \\ \vdots \\ x_{n,i}  r \mathbf{y}_j \end{array} \right]=
\X  (r E_{i,j}) \Y \in \X M_n(R) \Y \subseteq \Psi(N)=\left[ \begin{array}{cc} N \\ \vdots \\ N  \end{array} \right]$$
where $\mathbf{y}_j$ is the $j$-th row of $Y$. 
It follows that $x_{k,i} R \, \mathbf{y}_j \subseteq N$ for every $k,i,j=1,\ldots,n$. If 
$\Y \not\in \Psi(N)$ then $\mathbf{y}_j \not\in N$ for some $j$. 
It follows that $x_{k,i}R^n \subseteq N$ for every $k,i$. Therefore, $\X M_n(R) \subseteq \Psi(N)$.

Suppose that $I$ is a semiprime left ideal of $M_n(R)$. We will show that $\Phi(I)$ is a semiprime submodule of $R^n$. Pick $\mathbf{m} \in R^n$ such that $m_i R \mathbf{m} \subseteq  \Phi(I)$ for every $i$.
It follows that
$$\left[ \begin{array}{cc} \mathbf{m} \\ \mathbf{0} \\ \vdots \\ \mathbf{0} \end{array} \right]
M_n(R) \left[ \begin{array}{cc} \mathbf{m} \\ \mathbf{0} \\ \vdots \\ \mathbf{0} \end{array} \right] \subseteq
\left[ \begin{array}{cc} \mathbf{m} \\ \mathbf{0} \\ \vdots \\ \mathbf{0} \end{array} \right]
\left[ \begin{array}{cc} R \mathbf{m} \\ R\mathbf{m} \\ \vdots \\ R \mathbf{m} \end{array} \right] \subseteq
 \left[ \begin{array}{cc} \sum_{i=1}^n m_i R \mathbf{m} \\ \mathbf{0} \\ \vdots \\ \mathbf{0} \end{array} \right] \subseteq I $$
Since $I$ is semiprime, we get  $\left[ \begin{array}{cc} \mathbf{m} \\ \mathbf{0} \\ \vdots \\ \mathbf{0} \end{array} \right] \in I$ 
which implies that $\mathbf{m} \in \Phi(I)$.

Conversely, suppose that $N$ is a semiprime submodule of $R^n$. 
If $\Psi(N)$ is not semiprime, then there exists $\X \in M_n(R) \setminus \Psi(N)$ such that $\X M_n(R) \X \subseteq \Psi(N)$.
Pick $j$ such that $\mathbf{x}_j$, the $j$-th row of $\X$, does not belong to $N$. For every $i=1,\ldots,n$ and $r \in R$ we have that
$\X  (r E_{i,j}) \X \in \Psi(N)$. It follows that the $j$-th row of $\X (r E_{i,j}) \X$, which is equal to $x_{j,i} r \mathbf{x}_j$, belongs to $N$.
Since $N$ is semiprime, it follows that $\mathbf{x}_j \in N$, a contradiction. 
\end{proof}

\begin{rem}
\label{rem5}
Let $R$ be a commutative ring and $\mathfrak{p}$ a prime ideal of $R$.
%Let $\mathfrak{p}$ a two-sided prime ideal of $R$.
A left ideal $I$ of $M_n(R)$ corresponds to a $\mathfrak{p}$-prime submodule of $R^n$
iff $\{x \in R \mid x M_n(R) \subseteq I\}=\mathfrak{p}$ iff $\{\X \in M_n(R) \mid \X M_n(R) \subseteq I\}=M_n(\mathfrak{p})$.
A left ideal of $M_n(R)$ corresponds to a maximal $\mathfrak{p}$-prime submodule of $R^n$
iff it is of the form $D_{\mathfrak{p},\mathbf{u}}$ where $\mathbf{u} \in R^n \setminus \mathfrak{p}^n$.
\end{rem}

Theorem \ref{thm5} is the main result section.

\begin{thm}
\label{thm5}
Let $R$ be a commutative unital ring  and $n \in \NN$. 
\begin{enumerate}
\item Every semiprime left ideal of $M_n(R)$ is equal to the intersection of all prime left ideals  containing it.
\item If $R$ is a Jacobson ring then every prime left ideal of $M_n(R)$ is equal to the intersection of all maximal left ideals containing it.
\end{enumerate}
\end{thm}

\begin{proof}
The first part follows from Theorem \ref{thm2} and Proposition \ref{thm4}.
The second part follows from Propositions \ref{cor2} and \ref{thm4}.
\end{proof}

Theorem \ref{thm6} is an extension of Hilbert's Nullstellenstz to matrix polynomials.
It is the main application of our results.
It is a consequence of Theorem \ref{thm5} and Remark \ref{rem4}. 
We will give an alternative proof that uses Theorem \ref{thm3} and Proposition \ref{thm4}.

\begin{thm}
\label{thm6}
Let $R=k[x_1,\ldots,x_d]$ where $k$ is an algebraically closed field. For every $\G_1,\ldots,\G_m,\F \in M_n(R)$ the following are equivalent.
\begin{enumerate}
\item $\F$ belongs to the smallest semiprime left ideal of $M_n(R)$ that contains $\G_1,\ldots,\G_m$.
\item For every $\a \in k^d$ and every $\vv \in k^n$ such that $\G_1(a)\vv=\ldots=\G_m(a)\vv=0$ we have $\F(a)\vv=0$.
\end{enumerate}
\end{thm}

\begin{proof}
Let $I$ be the left ideal of $M_n(R)$ generated by $\G_1,\ldots,\G_m$
and let $\sqrt{I}$ be the smallest semiprime ideal of $M_n(R)$ containing $I$.
Then $\Phi(I)$ is the submodule of $R^n$ generated by all rows of all $\G_i$
and, by Proposition \ref{thm4}, $\Phi(\sqrt{I})$ is the smallest semiprime submodule of $R^n$ that contains $\Phi(I)$.
By Theorem \ref{thm3}, $\Phi(\sqrt{I})$ is equal to the intersection
of all submodules of the form $C_{\mathfrak{m}_\a,\vv}$ that contain $\Phi(I)$.
It follows that $\sqrt{I}$ is equal to the intersection of all submodules of 
the form $\Psi(C_{\mathfrak{m}_\a,\vv})$ that contain $I$. Finally, 
$\Psi(C_{\mathfrak{m}_\a,\vv})=D_{\mathfrak{m}_\a,\vv}=\{\G \in M_n(R) \mid \G(\a)\vv=0\}$ which completes the proof.
\end{proof}

\end{document}